\documentclass{amsart}
\usepackage{amssymb, amsfonts, amsthm, amsmath,tikz}
\usepackage[mathscr]{eucal}

\newtheorem{theorem}{Theorem}[section]
\newtheorem{thm}[theorem]{Theorem}
\newtheorem{IND}[theorem]{Inherent Nondualisability Lemma}
\newtheorem{IC}[theorem]{IC Duality Theorem}

\newtheorem{cor}[theorem]{Corollary}

\newtheorem{remark}{Remark}
\newtheorem{problem}[theorem]{Problem}

\newcommand{\up}[1]{\textup{#1}}

\title{Natural dualities, nilpotence and projective planes} 
\author{Marcel Jackson}
\address{Department of Mathematics and Statistics, La Trobe University, VIC 3086, Australia}
\email{m.g.jackson@latrobe.edu.au}
\begin{document}
\begin{abstract} 
We use an interpretation of projective planes to show the inherent nondualisability of some finite semigroups.  The method is sufficiently flexible to demonstrate the nondualisability of (asymptotically) almost all finite semigroups as well as to give a fresh proof of the Quackenbush-Szab\'o result that any finite group with a nonabelian Sylow subgroup is nondualisable.  A novel feature is that the ostensibly different notions of nilpotence for semigroups, nilpotence for groups, and the property of being nonorthodox for a completely 0-simple semigroup are unified by way of a single construction.   We also give a semigroup example of two dualisable finite semigroups whose direct product is inherently nondualisable.
\end{abstract}
\thanks{Results in this article were obtained over a eleven year period during which the author was supported by Australian Postdoctoral Fellowship DP0342459, ARC Discovery Project DP1094578 and ARC Future Fellowship FT1201000666.}
\maketitle

\section{Introduction}

The general theory of \emph{natural dualities} emerged from ``classical'' dualities such as Stone's duality for Boolean algebras, Pontryagin's duality for abelian groups and Priestley's duality for distributive lattices.  While these are instances of category theoretic dualities, they share enough common features at the algebra level to be treated concretely within a single algebraic framework.  This general theory of natural dualities was first developed by Davey and Werner in \cite{davwer} and has seen substantial development by many authors.  The standard reference is Davey and Clark \cite{cladav98}.  Most natural dualities concern quasivarieties generated by a single finite algebra, however the idea extends to relational structures, to quasivarieties generated by sets of algebras, and to quasivarieties generated by infinite algebras.  In general, the quasivariety of a finite algebra ${\bf M}$ may not admit a natural duality and in this case we will say that say that \emph{${\bf M}$ is nondualisable}.  If no finite algebra ${\bf N}$ whose quasivariety contains ${\bf M}$ is dualisable, then ${\bf M}$ is said to be \emph{inherently nondualisable}.  

The current article primarily concerns dualisability for finite semigroups.
A number of results have already been obtained in this area, though the results are more negative than positive.  On the positive side Al Dhamri~\cite{AlD}, has recently shown that every normal band is dualisable.  However in \cite{jac03} the author showed that any non-normal band is inherently nondualisable.  Quackenbush and Szab\'o \cite{quasza02b} showed that every finite  group with cyclic Sylow subgroups is dualisable, but also (\cite{quasza}) that a finite group with a nonabelian Sylow subgroup is inherently nondualisable.  There is currently no published proof that every finite group with abelian Sylow subgroups is dualisable, though such a proof has been announced by Nickedemus~\cite{nic}.

Recall that a \emph{Clifford semigroup} is a semilattice of groups: there is a congruence $\theta$ whose blocks are subgroups, and for which the corresponding quotient is a semilattice.  Positive dualisability results for groups can be extended to certain Clifford semigroups using results of Davey and Knox \cite{davkno}.  The result can be stated as follows: if the quasivariety of a finite group ${\bf G}$  admits a natural duality, then the quasivariety consisting of all Clifford semigroups whose subgroups lie in the quasivariety of ${\bf G}$ is finitely generated as a quasivariety and is dualisable (see \cite{jac03}).  In particular, if it is true that every finite group with abelian Sylow subgroups is dualisable, then a Clifford semigroup whose subgroups have abelian Sylow subgroups lies within a dualisable variety of Clifford semigroups.  In the world of finite monoids (in either the monoid signature or the semigroup signature) or the world of inverse semigroups (in either the unary semigroup signature or the semigroup signature), there is a converse: if a member of one of these classes has a subgroup with a nonabelian Sylow subgroup, or has a subalgebra that is not a Clifford semigroup, then the member is inherently nondualisable~\cite{jac03}.  Problem \ref{problem1} in the present article asks for an understanding of when Clifford semigroups are dualisable.

The present article will add to this list of mostly negative results.  We give a general nondualisability result (Theorem \ref{thm:T}), based on notions of nilpotence in semigroups.  This result is then applied to give our main results: Theorem~\ref{thm:nilvar}, which states that any finite semigroup whose \emph{variety} contains a proper $3$-nilpotent semigroup is inherently nondualisable;  a new proof of the Quackenbush and Szab\'o~\cite{quasza} result that any group with a nonabelian Sylow subgroup is inherently nondualisable (Theorem \ref{thm:group}); and Theorem~\ref{thm:cs}, which states that any completely simple semigroup that is not isomorphic to the direct product of a group with a rectangular band is inherently nondualisable.  There are a number of corollaries, including the result that the proportion of $n$-element semigroups that are dualisable approaches $0$ as $n$ tends to infinity (Corollary \ref{cor:almostall}), the result that a dualisable finite semigroup must have index $1$ (Corollary \ref{cor:index}) and that a dualisable finite regular semigroup must be completely regular and that each $\mathcal{J}$-class must be a direct product of a group with a rectangular band (Corollary \ref{cor:regular}).  We also show that the class of finite semigroups admitting a natural duality fails to be closed under finite direct products, by giving two 3-element semigroups that are dualisable but whose direct product is inherently nondualisable.

A number of these results were obtained in 2003 but despite being widely distributed amongst the algebra group at La Trobe University, they  managed to evade publication until this celebration of Brian Davey's 65th birthday.  A stumbling block was Theorem \ref{thm:nilvar}, which (along with Theorem \ref{thm:cs}) existed only the weaker form of Theorem \ref{thm:nilvarbinar} and Corollary \ref{cor:index} until early 2014.

\section{Duality  and nondualisability techniques}
The present article will focus on nondualisability, but for context and completeness we give a very brief overview of what it means to admit a natural duality.  The reader is directed to Clark and Davey \cite{cladav98} for a far more complete introduction to the topic.  A ``natural duality'' is a particular form of category-theoretic duality between a quasivariety of algebras (and more generally of structures) and a topological quasivariety.  For a quasivariety $\mathcal{Q}$ generated by a single finite algebra ${\bf M}$ (so, $\mathcal{Q}=\mathsf{ISP}({\bf M})$), the topological quasivariety will be generated by a different structure $\mathbb{M}$ on the same underlying universe $M$ of ${\bf M}$.  This ``alter ego''  $\mathbb{M}$ will carry the discrete topology, but also  operations, partial operations and relations.  In order for the alter ego $\mathbb{M}$ have any hope of facilitating a natural duality, it is necessary \cite[\S1.5]{cladav98} that each operation of $\mathbb{M}$ is a homomorphism from ${\bf M}^n$ into ${\bf M}$, that each relation  is a subalgebra of ${\bf M}^n$, and that each partial operation is a homomorphism from a subalgebra of ${\bf M}^n$ into ${\bf M}$ (here $n$ is the arity of the operation, relation or partial operation being considered).  The topological quasivariety $\mathcal{X}$ of $\mathbb{M}$ will be $\mathsf{IS_cP^+}(\mathbb{M})$, the class of all structures of the same type as $\mathbb{M}$ that arise by way of (topologically and algebraically) isomorphic copies of closed substructures of powers (with nonempty index sets) of $\mathbb{M}$, where topology is extended to powers by way of the product topology.  Under these assumptions, for each ${\bf A}\in \mathsf{ISP}({\bf M})$, the homset $\operatorname{hom}_\mathcal{Q}({\bf A},{\bf M})$ will always be a closed substructure of $\mathbb{M}^{A}$ and thus lies in $\mathcal{X}$.  This member of $\mathcal{X}$ is denoted by $D({\bf A})$.  Similarly, for each object $\mathbb{A}\in \mathcal{X}$ the homset $\operatorname{hom}_{\mathcal{X}}(\mathbb{A},\mathbb{M})$ (the set of continuous homomorphisms from $\mathbb{A}$ into $\mathbb{M}$) will be a subuniverse of ${\bf M}^A$; the corresponding subalgebra is denoted $E(\mathbb{A})$.  It is always the case that for every ${\bf A}\in\mathcal{Q}$, there is a natural evaluation map $e:{\bf A}\to E(D({\bf A}))$, given by $e_a(x)=x(a)$.  Under the existing assumptions on $\mathbb{M}$, this map is necessarily an injective homomorphism.  When $e$ is an isomorphism, then it is said that \emph{$\mathbb{M}$ yields a duality on ${\bf A}$}.  If $\mathbb{M}$ yields a duality on every ${\bf A}\in\mathcal{Q}$, then $\mathbb{M}$ is said to \emph{yield a \up(natural\up) duality on $\mathcal{Q}$}, and ${\bf M}$ is said to \emph{admit a natural duality} (by way of the alter ego $\mathbb{M}$) or be \emph{dualisable}.  Evidently, this is equivalent to each ${\bf A}\in\mathcal{Q}$ being isomorphic to a natural structure on the family of all continuous homomorphisms of an object in $\mathcal{X}$ into $\mathbb{M}$.  

The algebra ${\bf M}$ is \emph{inherently nondualisable} (or IND), if it does not lie in the quasivariety of any finite dualisable algebra.
 The standard tool for demonstrating inherent nondualisability is the following lemma from Clark and Davey \cite{cladav98} (see \cite[10.5.5]{cladav98}).  
   \begin{IND}\label{INDL} \cite{cladav98}
  Let $\mathbf{D}$ be a finite algebra. Then ${\bf D}$ is inherently nondualisable if there exists an infinite set $S$, a subalgebra $\mathbf{A}$ of $\mathbf{D}^{S}$ and an infinite subset $A_{0}$ of $A$ and a function $u\colon \mathbb{N}\to \mathbb{N}$ such that
 \begin{itemize}
 \item[(i)]  if $\mathrel{\theta}$ is a congruence on $\mathbf{A}$ of finite index at most $n$, then $\mathrel{\theta}\upharpoonright_{A_{0}}$ has only one class with more than $u(n)$ elements,
 \item[(ii)] $g\notin A$ where $g$ is the element of $D^{S}$ such that $g(s):=\rho_{s}(b)$ \up(the projection of $b$ to coordinate $s$\up), for each $s\in S$, with $b$ any element of the block of $\ker{\rho_{s}}{\upharpoonright_{A_{0}}}$ which has size greater than $u(|D|)$.
 \end{itemize}
  \end{IND}
The element $g$ in this lemma is usually known as the \emph{ghost element}.

We also make use of the following tool for demonstrating dualisability of a finite algebra ${\bf M}$.
\begin{IC} \cite[Corollary 2.2.12]{cladav98}\label{IC}
Suppose that $\mathbb{M}$ is an alter ego of ${\bf M}$.  Then $\mathbb{M}$ dualises ${\bf M}$ provided the following interpolation condition is satisfied\up:
for each $n\in\mathbb{N}$ and each substructure $\mathbb{X}\leq \mathbb{M}^n$, every morphism $\alpha:\mathbb{X}\to\mathbb{M}$ extends to term function $t:{\bf M}^n \to {\bf M}$ of the algebra ${\bf M}$.
\end{IC}

\section{Projective plane construction}
In this section we use the Inherent Nondualisability Lemma \ref{INDL} (henceforth, the IND Lemma) to give a general configuration causing inherent nondualisability.  The argument will apply to algebras in which there is a binary term operation $\cdot$, and elements $\mathsf{a},\mathsf{b},\mathsf{c},\mathsf{d},\mathsf{e},\mathsf{f}$ (not necessarily distinct) such that the following template ${\bf T}$ of products occur:
\begin{center}
\begin{tabular}{c|cc}
$\cdot$&$\mathsf{c}$&$\mathsf{d}$ \\
\hline
$\mathsf{a}$& $\mathsf{e}$ &$\mathsf{f}$  \\
$\mathsf{b}$& $\mathsf{f}$ &$\mathsf{f}$  \\
\end{tabular}
\end{center}
An algebra in which this occurs is said to \emph{interpret ${\bf T}$}.  Our main result will require an interpretation of ${\bf T}$ with $\mathsf{e}\neq\mathsf{f}$, as well as a more technical ``geometric'' condition that we describe in due course.  A very large array of algebraic structures interpret ${\bf T}$ with $\mathsf{e}\neq\mathsf{f}$, though many will fail the technical condition.

As an example, consider any algebra in which there is a fundamental binary operation $\cdot$ for which there is a multiplicative $0$ but such that not every product (in $\cdot$) equals $0$.  So there are elements $\mathsf{a}$ and $\mathsf{c}$ (possibly equal) such that $\mathsf{e}:=\mathsf{a}\cdot \mathsf{c}\neq 0$.  Letting $\mathsf{b}=\mathsf{d}=\mathsf{f}:=0$ we obtain an interpretation of ${\bf T}$ with $\mathsf{e}\neq \mathsf{f}$.  We will see that asymptotically, almost all finite semigroups interpret ${\bf T}$ in this way, as well as satisfying the additional geometric condition.

As a second example consider any nonabelian group ${\bf G}$, so that there are elements $\mathsf{a}$ and $\mathsf{c}$ such that the commutator $\mathsf{e}:=[\mathsf{a},\mathsf{c}]\neq 1$.  Then ${\bf G}$ interprets ${\bf T}$ via the binary term operation of commutator by letting $\mathsf{b}=\mathsf{d}=\mathsf{f}:=1$.  The technical condition will be shown to hold when ${\bf G}$ contains an abelian Sylow subgroup.

Recall that a projective plane $\mathcal{P}$ consists of a set of points $\mathsf{P}$ along with a family $L$ of sets of points, known as lines, satisfying the property that any pair of distinct points are members of a unique line, any pair of distinct lines intersect to a unique point, and there are four points in general position: no three lying on the same line.  Lines will be denoted by upper case $L,K$, possibly with subscripts.  For two distinct lines $L,K$ we let $L\wedge K$ denote the unique point on both $L$ and $K$; otherwise, if $L=K$ then $L\wedge K=L$.  For two distinct points $p,q$, the unique line containing both $p$ and $q$ is denoted by $p\vee q$.

We will assume throughout that any projective plane we consider has infinitely many points.  It is well known that there exist projective planes of all infinite cardinalities (measured in terms of the cardinality of the set of points); this follows from applications of the L\"owenheim Skolem Theorems for example, or from direct constructions based over fields.    While we do not make explicit use of this in the article, it can be used (by trivial adjustments to assumptions) to push the inherent nondualisability results in this article to proofs of inherent non-$\kappa$-dualisability in the sense of Davey, Idziak, Lampe and McNulty \cite{DILM}.

Let $\infty$ be a symbol not in $\mathsf{P}$, and let $\mathsf{P}_\infty$ denote $\mathsf{P}\cup\{\infty\}$.    The ``point'' $\infty$ is used as a book-keeping device, to record the ghost element.

We now describe some standard notation for certain elements of cartesian powers.  Let $M,S$ be sets, and $M^S$ be the usual cartesian power.  If $I_1,\dots,I_n$ are pairwise disjoint subsets of $S$ and $a,b_1,\dots,b_n$ are elements of $M$, then $a_{I_1,\dots,I_n}^{b_1,\dots,b_n}$ denotes the element of $M^S$ given by
\[
a_{I_1,\dots,I_n}^{b_1,\dots,b_n}(i)=\begin{cases}
b_j&\text{ if }i\in I_{j}\\
a&\text{ otherwise}.
\end{cases}
\]
In a slight abuse of notation, we allow elements of $S$ in the subscript to be considered as if they were singleton sets.  For example, if $S=\mathbb{Z}$ (the integers), then $a_{1,2\mathbb{Z}}^{b,c}$ is the same as $a_{\{1\},2\mathbb{Z}}^{b,c}$, the tuple that is $c$ on all even coordinates and $a$ on all odd coordinates except for coordinate $1$ where it equals $b$.

\begin{thm}\label{thm:T}
Let ${\bf M}$ be an algebra interpreting ${\bf T}$ via some binary term operation.  If the subalgebra of ${\bf M}^{\mathsf{P}_\infty}$ generated by the following ``\emph{line generators}''\up:
\[
\{\mathsf{b}_{\infty,L}^{\mathsf{a},\mathsf{a}},\mathsf{d}_{\infty,L}^{\mathsf{c},\mathsf{c}}\mid L \text{ a line in }\mathcal{P}\}  
\]
does not contain the ghost element $g:=\mathsf{f}_{\infty}^\mathsf{e}$, then ${\bf M}$ is inherently nondualisable.
\end{thm}
\begin{proof}
We apply the Inherent Nondualisability Lemma \ref{INDL} (the IND Lemma) to the subalgebra ${\bf A}$ of ${\bf M}^{\mathsf{P}_\infty}$ generated by the line generators.  The set $A_0$ of the IND Lemma is chosen as $A_0:=\{\mathsf{f}_{\infty,p}^{\mathsf{e},\mathsf{e}}\mid p\in \mathsf{P}\}$, so that the ghost element $g$ is $\mathsf{f}_{\infty}^\mathsf{e}$, in agreement with the statement of the present theorem.  Note that if $L$ and $K$ are lines with $L\wedge K=p$, then $\mathsf{f}_{\infty,p}^{\mathsf{e},\mathsf{e}}$ arises in ${\bf A}$ by way of $\mathsf{b}_{\infty,L}^{\mathsf{a},\mathsf{a}}\mathsf{d}_{\infty,K}^{\mathsf{c},\mathsf{c}}=\mathsf{f}_{\infty,L\wedge K}^{\mathsf{e},\mathsf{e}}=\mathsf{f}_{\infty,p}^{\mathsf{e},\mathsf{e}}$, so that $A_0$ is indeed a subset of ${\bf A}$.  Item (ii) of the IND Lemma holds by assumption.  We need to establish item (i).

We consider a congruence $\theta$ on ${\bf A}$ of index strictly less than $n$ (so that the function $u:\mathbb{N}\to\mathbb{N}$ of the IND Lemma is given by $u(n):=n+1$).  We show that the restriction of $\theta$ to $A_0$ has at most one block of size $n$ or more.  For this, assume that $\{p_1,\dots,p_n\}$ and $\{q_1,\dots,q_n\}$ are disjoint $n$-element subsets of the plane $\mathcal{P}$ such that for each $i,j\leq n$ we have $\mathsf{f}_{\infty,p_i}^{\mathsf{e},\mathsf{e}}\mathrel{\theta}\mathsf{f}_{\infty,p_j}^{\mathsf{e},\mathsf{e}}$ and $\mathsf{f}_{\infty,q_i}^{\mathsf{e},\mathsf{e}}\mathrel{\theta}\mathsf{f}_{\infty,q_j}^{\mathsf{e},\mathsf{e}}$.  Our goal is to show that $\mathsf{f}_{\infty,p_i}^{\mathsf{e},\mathsf{e}}\mathrel{\theta}\mathsf{f}_{\infty,q_j}^{\mathsf{e},\mathsf{e}}$, so that there is just one ``large'' block of $\theta$ on $A_0$.  

Let $r$ be any  point not amongst $p_1,\dots,p_n,q_1,\dots,q_n$ and such that the $n$ lines of the form $L_i:=p_i\vee r$ are pairwise distinct and do not contain $q_j$ for any $j\leq n$; see Figure \ref{fig:1}.
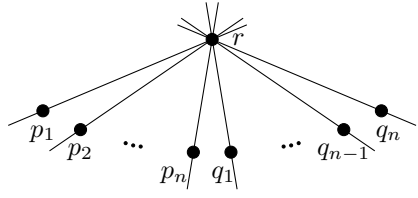
\begin{figure}
\begin{tikzpicture}
\draw [fill] (0,0.55) circle [radius=0.08];
\node at (0,0.25) {$p_1$};
\draw [fill] (0.5,0.3) circle [radius=0.08];
\node at (0.5,0) {$p_2$};

\draw [fill] (1.1,0.12) circle [radius=0.02];
\draw [fill] (1.2,0.09) circle [radius=0.02];
\draw [fill] (1.3,0.075) circle [radius=0.02];

\draw [fill] (2,0) circle [radius=0.08];
\node at (1.75,-.3) {$p_n$};
\draw [fill] (2.5,0) circle [radius=0.08];
\node at (2.4,-.3) {$q_1$};

\draw [fill] (3.4,0.12) circle [radius=0.02];
\draw [fill] (3.3,0.09) circle [radius=0.02];
\draw [fill] (3.2,0.075) circle [radius=0.02];

\draw [fill] (4,0.3) circle [radius=0.08];
\node at (4,0) {$q_{n-1}$};
\draw [fill] (4.5,0.55) circle [radius=0.08];
\node at (4.6,0.25) {$q_{n}$};

\draw [fill] (2.25,1.5) circle [radius=0.08];
\node at (2.6,1.5) {$r$};

\draw (-0.461,0.356) -- (2.711,1.694);
\draw (0.0876,0.0172) -- (2.662,1.783);
\draw (1.918,-0.493) -- (2.332,1.993);
\draw (2.5814,-0.493) -- (2.169,1.993);
\draw (4.412,0.0172) -- (1.838,1.783);
\draw (4.961,0.356) -- (1.789,1.694);

\end{tikzpicture}\caption{Points $p_1,\dots,p_n,q_1,\dots,q_n$ and a choice of $r$.}\label{fig:1}
\end{figure}
As $\theta$ has index less than $n$ it follows that there are $i\neq j$ (both at most $n$) such that $\mathsf{b}_{\infty,L_i}^{\mathsf{a},\mathsf{a}}\mathrel{\theta}\mathsf{b}_{\infty,L_j}^{\mathsf{a},\mathsf{a}}$.  Fixing such a choice of $i,j$, for every point $p\in L_i\backslash\{r\}$ we have 
\begin{align*}
\mathsf{f}_{\infty,p}^{\mathsf{e},\mathsf{e}}&=\mathsf{b}_{\infty,L_i}^{\mathsf{a},\mathsf{a}}\mathsf{d}_{\infty,p\vee p_j}^{\mathsf{c},\mathsf{c}}\\
&\mathrel{\theta}\mathsf{b}_{\infty,L_j}^{\mathsf{a},\mathsf{a}}\mathsf{d}_{\infty,p\vee p_j}^{\mathsf{c},\mathsf{c}}\\
&=\mathsf{f}_{\infty,p_j}^{\mathsf{e},\mathsf{e}}\\
&\mathrel{\theta}\mathsf{f}_{\infty,p_k}^{\mathsf{e},\mathsf{e}} \qquad (\text{for all }k=1,\dots,n).
\end{align*}
Thus,
\begin{equation}
\text{all points $p$ on $L_i$ except possibly $r$ have }\mathsf{f}_{\infty,p}^{\mathsf{e},\mathsf{e}}\mathrel{\theta}\mathsf{f}_{\infty,p_1}^{\mathsf{e},\mathsf{e}}.\tag{$\dagger$}\label{dagger}
\end{equation}
Continuing with the fixed choice of $i\leq n$, there are only finitely many points on $L_i$ of the form $(q_k\vee q_{k'})\wedge L_i$ (for some $k\neq k'$ in $\{1,\dots,n\}$).  As $\mathsf{P}$ is infinite, the number of points on $L_i$ is infinite, so we may select distinct points $r_1,\dots,r_n$  that are \emph{not} $r$ and are  \emph{not} of the form $(q_k\vee q_{k'})\wedge L_i$.  In particular,  for each $k_1\neq k_2$  the lines $r_{k_1}\vee q_{k_1}$ and $r_{k_2}\vee q_{k_2}$ are distinct; see Figure \ref{fig:2}.
\begin{figure}
\begin{tikzpicture}
\draw [fill] (0,0) circle [radius=0.08];
\node at (0.2,-0.3) {$p_i$};

\draw [fill] (1,2.5) circle [radius=0.08];
\node at (1.3,2.5) {$r$};

\draw [fill] (0.2,.5) circle [radius=0.08];
\node at (-0.1,.4) {$r_1$};
\draw [fill] (0.6,1.5) circle [radius=0.08];
\node at (0.1,1.4) {$r_{n-1}$};
\draw [fill] (0.8,2) circle [radius=0.08];
\node at (0.5,1.85) {$r_n$};

\draw (-0.371,-0.928) -- (1.371,3.428);
\node at (1,3.3) {$L_i$};

\draw [fill] (2.5,0) circle [radius=0.08];
\node at (2.4,-.3) {$q_1$};

\draw [fill] (3.4,0.57) circle [radius=0.02];
\draw [fill] (3.3,0.49) circle [radius=0.02];
\draw [fill] (3.2,0.42) circle [radius=0.02];

\draw [fill] (4,1) circle [radius=0.08];
\node at (4.5,0.8) {$q_{n-1}$};
\draw [fill] (4.5,1.5) circle [radius=0.08];
\node at (4.9,1.3) {$q_{n}$};

\draw (2.989,-0.106) -- (-0.289,0.606);
\draw (4.495,0.927) -- (0.105,1.573);
\draw (4.995,1.433) -- (0.305,2.067);

\end{tikzpicture}\caption{The line $L_i$, the points $r_1,\dots,r_n$ and connecting lines.}\label{fig:2}
\end{figure}
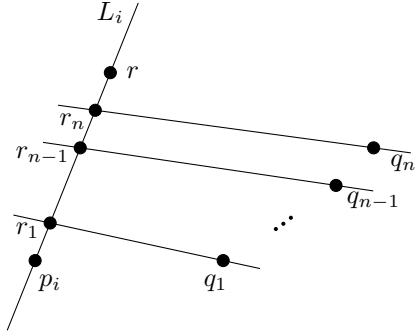
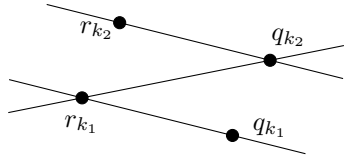
\begin{figure}
\begin{tikzpicture}
\draw [fill] (0,0.5) circle [radius=0.08];
\node at (0,0.15) {$r_{k_1}$};

\draw [fill] (0.5,1.5) circle [radius=0.08];
\node at (0.2,1.4) {$r_{k_2}$};

\draw [fill] (2,0) circle [radius=0.08];
\node at (2.5,0.1) {$q_{k_1}$};

\draw [fill] (2.5,1) circle [radius=0.08];
\node at (2.75,1.3) {$q_{k_2}$};

\draw (-0.97,0.743) -- (2.97,-0.243);
\draw (-0.47,1.743) -- (3.47,0.757);
\draw (-0.981,0.304) -- (3.481,1.196);

\end{tikzpicture}\caption{The lines $r_{k_1}\vee q_{k_2}$, $r_{k_1}\vee q_{k_1}$ and $r_{k_2}\vee q_{k_2}$.}\label{fig:3}
\end{figure}

As the index of $\theta$ is less than $n$, it follows that there are $k_1\neq k_2$ such that $\mathsf{d}_{\infty,r_{k_1}\vee q_{k_1}}^{\mathsf{c},\mathsf{c}}\mathrel{\theta}\mathsf{d}_{\infty,r_{k_2}\vee q_{k_2}}^{\mathsf{c},\mathsf{c}}$.  Then using $r_{k_1}=(r_{k_1}\vee q_{k_2})\wedge (r_{k_1}\vee q_{k_1})$ and $q_2=(r_{k_1}\vee q_{k_2})\wedge (r_{k_2}\vee q_{k_2})$
(see Figure \ref{fig:3}) we have
\begin{align*}
\mathsf{f}_{\infty,p_1}^{\mathsf{e},\mathsf{e}}&\mathrel{\theta} \mathsf{f}_{\infty,r_{k_1}}^{\mathsf{e},\mathsf{e}}\qquad (\text{by \eqref{dagger}})\\
&=\mathsf{b}_{\infty,r_{k_1}\vee q_{k_2}}^{\mathsf{a},\mathsf{a}}\mathsf{d}_{\infty,r_{k_1}\vee q_{k_1}}^{\mathsf{c},\mathsf{c}}\\
&\mathrel{\theta}\mathsf{b}_{\infty,r_{k_1}\vee q_{k_2}}^{\mathsf{a},\mathsf{a}}\mathsf{d}_{\infty,r_{k_2}\vee q_{k_2}}^{\mathsf{c},\mathsf{c}}\\
&=\mathsf{f}_{\infty,q_{k_2}}^{\mathsf{e},\mathsf{e}},
\end{align*}
which completes the proof that there is exactly one block of $\theta$ with size $n$ or more.
\end{proof}
\section{Semigroup theoretic preliminaries}\label{sec:sgp}
In this section we set some basic notation and recall some of the fundamental structural theory of finite semigroups.  The reader is directed to a text such as Howie \cite{how} for a full treatment; material on Rees matrix semigroups may be found in \cite[Chapter 3]{how} for example.

Recall that the \emph{index} of a finite semigroup ${\bf S}$ is the smallest number $i$ such that ${\bf S}\models x^{i}\approx x^{i+p}$ for some $p>1$.  The smallest number $p$ for which this equation holds is called the \emph{period}.  If $d\geq i$ is such that $d$ is congruent to $0$ modulo $p$  then $s^d s^d=s^{d+d}=s^d$ for any $s\in S$.   The element $s^d$ also arises as the $\lim_{n\to \infty}s^{n!}$, which is eventually constant in any finite semigroup.  The notation $s^\omega$ is usually used to denote this idempotent power and the notation extends to $s^{\omega+i}$ for any $i\in\mathbb{Z}$ by setting $s^{\omega+i}$ to be $s^\omega s^{i'}$ where $i'$ is any positive integer congruent to $i$ modulo $p$.  Note that $s^{\omega+i}s^{\omega+j}=s^{\omega+i+j}$.

A semigroup is \emph{$k$-nilpotent} if it satisfies the semigroup law $x_1x_2\dots x_k\approx y_1y_2\dots y_k$, which is equivalent to the property that there is a zero element $0$ and every product of length $k$ is equal to $0$.  A $k$-nilpotent semigroup is a \emph{proper} $k$-nilpotent semigroup if it is not $(k-1)$-nilpotent.  Semigroups that are 2-nilpotent are often called null semigroups.

A semigroup ${\bf S}$ is said to be \emph{simple} if the only ideal of ${\bf S}$ is ${\bf S}$ itself: this is not the same as the universal algebraic notion of being simple, which in semigroup theory is usually called \emph{congruence free}.  A semigroup ${\bf S}$ is \emph{$0$-simple} if it has a 0 element and the only ideals are $\{0\}$ and ${\bf S}$.  A ($0$-)simple semigroup is \emph{completely} ($0$-)simple if it has a primitive idempotent; that is, an idempotent $e$ such that whenever $f$ is idempotent with $ef=fe=f\neq 0$ then $e=f$.  It is a classical result of semigroup theory that all finite ($0$-)simple semigroups are completely ($0$-)simple.

We now recall Rees' powerful classification theorem for completely ($0$)-simple semigroups.  Choose a group ${\bf G}$ and let $0$ be a symbol not in $G$.  Choose a pair of nonempty sets $I,\Lambda$ and a $\Lambda\times I$ matrix $P$ with entries from $G\cup\{0\}$ such that no row nor column consists entirely of $0$.  (Note that $P$ is notationally distinct from the projective plane $\mathcal{P}$ on points $\mathsf{P}$.)  The \emph{Rees matrix semigroup with $0$} built from ${\bf G}$ and $P$, denoted $M^0[{\bf G},P]$, is the semigroup on the universe $\{0\}\cup \{(i,g,\lambda)\mid i\in I,g\in G,\lambda\in \Lambda\}$ with multiplication
\[
(i,g,\lambda)(j,h,\rho):=\begin{cases}
(i,gP_{\lambda,j}h,\rho)&\text{ if }P_{\lambda,j}\in G\\
0&\text{ otherwise.}
\end{cases}
\]
where $P_{\lambda,j}$ is the $(\lambda,j)^{\rm th}$ entry of $P$.  Rees matrix semigroups with $0$ are always completely $0$-simple, and moreover every completely $0$-simple semigroup arises in this way.  If all entries of $P$ are from $G$, then the element $0$ may be dropped (the notation is $M[{\bf G},P]$) and one obtains a construction for completely simple semigroups.  It is possible for different matrices $P$ to give rise to the same semigroup, up to isomorphism.

Completely $0$-simple semigroups form a basic building block of any finite semigroup.  Recall that in a semigroup ${\bf S}$ we say that $a$ \emph{divides} $b$ if there are elements $c$ or $d$ in $S$ (or possibly empty) such that $cad=b$.  The ``divides'' relation defines a preorder on any semigroup, and the equivalence classes are known as \emph{$\mathcal{J}$-classes}.
\begin{thm}\label{thm:Rees}
Let ${\bf S}$ be a finite semigroup,  $s$ be an element of $S$ and $J_s$ denote the $\mathcal{J}$-class of $s$.  Let ${\bf J}$ be the semigroup generated by $J_s$ and $I$ be the ideal consisting of all elements of $J$ not in $J_s$.  If $I$ is empty, then ${\bf J}$ is a completely simple semigroup.  If $I$ is nonempty, and $J_s$ contains an idempotent, then ${\bf J}/I$ is a completely $0$-simple semigroup.  If $I$ contains no idempotent, then ${\bf J}/I$ is a null semigroup\up: all products equal $0$.
\end{thm}
Note that when $J_s$ is the minimum ideal for example, then $I$ is empty, so that the minimum ideal of a finite semigroup is always a completely simple semigroup (of course it may be a degenerate, such as a single multiplicative $0$).  

Theorem \ref{thm:Rees} is one of the fundamental tools in semigroup theory.  We use it here to illustrate some further basic fact that will be used later in the article. The facts are well known to researchers in semigroup varieties, though the author is not aware of a location where they have been spelt out explicitly.

\begin{thm}\label{thm:c0s}
A completely $0$-simple semigroup $M^0[{\bf G},P]$ generates a variety containing a proper $3$-nilpotent semigroup if and only if $P$ contains a $0$ entry.
\end{thm}
\begin{proof}
If $P$ contains no nonzero elements then $M^0[{\bf G},P]$ satisfies $x^{p+1}\approx x$, where $p$ is the exponent of the group ${\bf G}$.  This law fails on any proper $3$-nilpotent semigroup.  

Now assume that $P$ contains a $0$ entry, say $P_{i,\lambda}\in G$.   Factor by the congruence whose equivalence classes are $\{0\}$ along with the blocks $\{(i,g,\lambda)\mid g\in G\}$ for each $(i,\lambda)\in I\times \Lambda$.  Then we have a semigroup isomorphic to $M^0[{\bf 1},P']$ where ${\bf 1}$ is the one element group on $\{1\}$, and $P'$ is the matrix $P$ with all nonzero elements replaced by $1$.  Now let $P_{\lambda,i}$ be a nonzero entry in $P$.  Let $j,\rho$ be such that $P_{\rho,i}$ and $P_{\lambda,j}$ are nonzero, which exists because each row and column of $P$ has a nonzero entry.  Let $e$ denote the element $(i,1,\rho)$, $f$ denote $(j,1,\lambda)$ and $a$ denote $(i,1,\lambda)$.  Then $ea=a=af$, while $ae=aa=0$.  Also $ee=e$ and $ff=f$.    In the square $M^0[{\bf 1},P']\times M^0[{\bf 1},P']$, consider the subsemigroup ${\bf M}$ generated by $(e,a),(a,f)$.  Now $(e,a)(a,f)=(a,a)$, but because $ae=aa=0$, every other product produces a tuple with $0$ in a coordinate.  These elements with a $0$ coordinate form an ideal $J$, and  ${\bf M}/J$ is a proper $3$-nilpotent semigroup.
\end{proof}
Recall that a semigroup is \emph{regular} if for every $s$ there is a $t$ such that $sts=s$.  This is equivalent to every $\mathcal{J}$-class containing an idempotent (cf.~Theorem \ref{thm:Rees}).  A semigroup is \emph{completely regular} if every element lies within a subgroup, which is equivalent to every $\mathcal{J}$-class being a completely simple semigroup.
\begin{thm}\label{thm:regular}
A finite regular semigroup generates a variety containing a proper $3$-nilpotent semigroup if and only if it is not completely regular.
\end{thm}
\begin{proof}
If ${\bf S}$ is completely regular of period $p$, then it satisfies $x^{p+1}\approx x$, which fails on any proper $3$-nilpotent semigroup.  Now assume that ${\bf S}$ is a finite regular semigroup containing an element $s$ with $s^{p+1}\neq s$.  Then $s^2$ is not in the $\mathcal{J}$-class of $s$, showing that the quotient ${\bf J}/I$ of Theorem \ref{thm:Rees} (built from $J_s$) is a completely $0$-simple semigroup.  Because $s^2=0$ in ${\bf J}/I$ it follows that when presented as a Rees matrix semigroup $M^0[{\bf G},P]$, the matrix $P$ must contain a zero entry.
\end{proof}

\section{Nilpotent and monogenic semigroups}\label{sec:nil}
While $2$-nilpotent semigroups do not exhibit any interesting properties, proper $3$-nilpotent semigroups present an interesting jump in the complexity of certain algebraic properties: a proper $3$-nilpotent semigroup generates a residually large variety (Golubov and Sapir \cite{golsap}, Kublanovski \cite{kub}, or McKenzie \cite{mck}); no proper $3$-nilpotent semigroup has a finite basis for its quasi-identities (Jackson and Volkov \cite{jacvol09}); and $3$-nilpotent semigroups provide a key role in undecidability results relating to membership problems \cite{HKMST,jacvol09b}.  We now add to this list of  complex behaviour by showing that they are inherently nondualisable.  We mention that the template ${\bf T}$ originated in the proof of early versions of the following theorem, using the notion of a ``homotopy'', in the style of \cite{jacvol09}.

\begin{theorem}\label{thm:nilvar}
Let ${\bf S}$ be a finite semigroup whose variety contains a proper $3$-nilpotent semigroup.  Then ${\bf S}$ is inherently nondualisable.
\end{theorem}
\begin{proof}
Let ${\bf N}$ be a proper $3$-nilpotent semigroup in $\mathsf{HSP}({\bf S})$, and that $a,c,e\in N$ are such that $ac=e\neq 0$ (possibly $a=c$ however we must have $e\notin\{a,c\}$ as $e\neq 0$, and $e=a$ implies $acc=ec=ac=e\neq 0$, with a similar contradiction to $3$-nilpotence if $e=c$).  We may assume without loss of generality that ${\bf N}$ is generated by $a,c$.  Let ${\bf M}\in\mathsf{SP}({\bf S})$ be such that there is a surjective homomorphism $\nu:{\bf M}\to {\bf N}$.  By the definition of inherent nondualisability, it will suffice to show that ${\bf M}$ is inherently nondualisable.

Select any $\mathsf{a}\in\nu^{-1}(a)$ and $\mathsf{c}\in\nu^{-1}(c)$ and set $\mathsf{e}:=\mathsf{a}\mathsf{c}\in\nu^{-1}(e)$.  Because ${\bf N}$ is finite we may assume that ${\bf M}$ is finite, and moreover, because ${\bf N}$ is generated by $a,c$ we may assume that ${\bf M}$ also is generated by $\mathsf{a}$ and $\mathsf{c}$.  Let $I$ be the minimum ideal of ${\bf M}$.  As ${\bf M}$ is generated by $\mathsf{a},\mathsf{c}$, there is a word $u$ in the alphabet $\{\mathsf{a},\mathsf{c}\}$ such that the product $u$ lies in $I$.  Now observe that $(\mathsf{a}^\omega u \mathsf{c}^\omega)^\omega$ is also in $I$ and is additionally an idempotent element.  We let this element be denoted by $v$.  Now as $\mathsf{a}^\omega$ and $\mathsf{c}^\omega$ are idempotents that occur at either end of the product $v$, we have $v=\mathsf{a}^\omega v=v\mathsf{c}^\omega$.  Thus $\mathsf{a}v=\mathsf{a}\mathsf{a}^\omega v=\mathsf{a}^{\omega+1} v$ and $v\mathsf{c}=v\mathsf{c}^\omega \mathsf{c}=v\mathsf{c}^{\omega+1}$, giving
\begin{equation}
\mathsf{a}^{\omega+1}v\mathsf{c}^{\omega+1}=
\mathsf{a}v\mathsf{c}^{\omega+1}=\mathsf{a}^{\omega+1}v\mathsf{c}=\mathsf{a}v\mathsf{c}\tag{$\star$}\label{avc}
\end{equation}

Now we may complete the interpretation of ${\bf T}$ into ${\bf M}$.  To the existing choice of $\mathsf{a},\mathsf{c},\mathsf{e}$ with $\mathsf{a}\mathsf{c}=\mathsf{e}$, add $\mathsf{b}:=\mathsf{a}^{\omega+1}v$ and $\mathsf{d}:=v\mathsf{c}^{\omega+1}$ and $\mathsf{f}:=\mathsf{a}v\mathsf{c}$.  To see that this is a valid interpretation note that using Equation \eqref{avc} and the idempotence of $v$ we have 
\begin{align*}
\mathsf{a}\cdot \mathsf{d}&=\mathsf{a}v\mathsf{c}^{\omega+1}=\mathsf{f}\\
\mathsf{b}\cdot\mathsf{c}&=\mathsf{a}^{\omega+1}v\mathsf{c}=\mathsf{f}\text{ and }\\
\mathsf{b}\cdot\mathsf{d}&=\mathsf{a}^{\omega+1}vv\mathsf{c}^{\omega+1}=\mathsf{a}^{\omega+1}v\mathsf{c}^{\omega+1}=\mathsf{f}.
\end{align*}
Now we need to verify that the ghost element $g$ of Theorem \ref{thm:T} cannot be generated by the line generators in ${\bf M}^{\mathsf{P}_\infty}$.  First let $J$ denote $\nu^{-1}(0)$.  Next, let $h$ be any element of ${\bf M}^{\mathsf{P}_\infty}$ that is generated by line generators and which has $h(\infty)=\mathsf{e}$.  We show that there is a point $q$ in $\mathsf{P}$ such that $h(q)=\mathsf{e}$ also, showing that $g\neq h$.  

Now $h$ must arise as a product of exactly two line generators.  This is because at the coordinate $\infty$, the line generators equal either $\mathsf{a}$ or $\mathsf{c}$.  Because any product of length $3$ or more in ${\bf N}$ equals $0$, so too must any product of length three or more lie in the ideal $J$ of ${\bf M}$.  Because $h(\infty)\notin J\cup\{\mathsf{a},\mathsf{c}\}$ it arises as a product of exactly two line generators.  Let $L,K$ be the lines corresponding to the two line generators, and let $q$ be any point on $L\wedge K$.  Then $h(q)=h(\infty)$ as claimed.  Thus Theorem \ref{thm:T} applies to show that ${\bf M}$ is inherently nondualisable as required.
\end{proof}
It would be interesting if the requirement of associativity of ${\bf S}$ in this theorem could be dropped.  In the current proof, associativity is being used heavily to identify the elements $\mathsf{b},\mathsf{d},\mathsf{f}$.  For a general finite binar (algebra with single binary operation), one can identify the corresponding notion of ``minimum ideal'' $I$ and identify some product---now requiring bracketing---lying in it \cite{jactro}.  But it is not obvious how to obtain something like the equalities in \eqref{avc}, which appear to require something like associativity.  The next theorem circumvents this by assuming that there is a minimum ideal that behaves nicely.
\begin{theorem}\label{thm:nilvarbinar}
Let ${\bf B}$ be a finite binar with a multiplicative $0$ element and such that the variety generated by ${\bf B}$ contains a proper $3$-nilpotent semigroup.  Then ${\bf B}$ is inherently nondualisable.
\end{theorem}
\begin{proof}
The proof is very similar to that of Theorem \ref{thm:nilvar}, so we give only a sketch.
Let ${\bf N}$ be a proper $3$-nilpotent semigroup in $\mathsf{HSP}({\bf B})$, and that $a,c,e\in N$ are such that $ac=e\neq 0$.  As in the proof of Theorem \ref{thm:nilvar} we may assume that ${\bf N}$ is generated by $a,c$, that there is a finite ${\bf M}\in\mathsf{SP}({\bf B})$, a  surjective homomorphism $\nu:{\bf M}\to {\bf B}$ and element $\mathsf{a}\in\nu^{-1}(a)$, $\mathsf{c}\in\nu^{-1}(c)$ with ${\bf M}$ generated by $\mathsf{a},\mathsf{c}$.  In a slight deviation to the proof of Theorem \ref{thm:nilvar}, observe that if the tuple $(0,\dots,0)$ is not already in $M$, then it may be added, and the homomorphism $\nu$ extended by setting $\nu((0,\dots,0)):=0$.  Thus we now adjust ${\bf M}$ if necessary, by assuming that it does contain $(0,\dots,0)$.  Now let $\mathsf{b}=\mathsf{d}=\mathsf{f}:=(0,\dots,0)$.  This gives an interpretation of ${\bf T}$, and the remainder of the argument is essentially a simplified version of the final stages of the application of Theorem \ref{thm:T} in the proof of Theorem~\ref{thm:nilvar}.
\end{proof}
Theorem \ref{thm:nilvarbinar} implies Theorem 2.5 of \cite{BDPW} (a fact that is alluded to in Remark 2.6 of \cite{BDPW}).  Conversely, it appears that some cases where we apply Theorem \ref{thm:T} can alternatively be obtained using embellishments of the constructions in the proof of \cite[Theorem 2.5]{BDPW}, though not when, for example, $\mathsf{a}=\mathsf{c}$.

To finish this section we give some corollaries to Theorem \ref{thm:nilvar}.
\begin{cor}\label{cor:nil}
If ${\bf S}$ is a finite proper $k$-nilpotent semigroup for $k>2$, then ${\bf S}$ is inherently nondualisable.
\end{cor}
\begin{proof}
Let $I$ be the ideal of ${\bf S}$ consisting of all elements that cannot be written as a product of length $3$.  Then ${\bf S}/I$ is a $3$-nilpotent semigroup.  It is a \emph{proper} $3$-nilpotent semigroup because, by assumption, there are elements  $a_1,\dots,a_{k-1}$ such that $a_1\dots a_{k-1}\neq 0$.  Then the element $a_1a_2$ cannot be equal to a product $b_1b_2b_3$ because of the contradiction $0=b_1b_2b_3a_3\dots a_{k-1}=a_1\dots a_{k-1}\neq 0$.  Thus ${\bf S}$ is inherently nondualisable by Theorem \ref{thm:nilvar}.
\end{proof}

Kleitman, Rothschild and Spencer \cite{klerotspe} showed that the proportion of all $n$-element semigroups that are proper 3-nilpotent approaches $1$ as $n\to\infty$.  Thus we obtain the following corollary to Theorem \ref{thm:nilvar}.
\begin{cor}\label{cor:almostall}
As $n\to\infty$, the proportion of all $n$-element semigroups which are dualisable approaches $0$.
\end{cor}

\begin{cor}\label{cor:index}
A semigroup is inherently nondualisable if it has index more than~$2$.
\end{cor}
\begin{proof}
Let ${\bf S}$ be a finite semigroup with index $i>2$ and period $p$.  Thus there is an element $a\in S$ such that $a^i=a^{i+p}$ but $a^{i-1}\neq a^{i-1+p}$.  Consider the subsemigroup $\langle a\rangle$ generated by $a$, and factor by the ideal $\{a^j\mid j>2\}$.  This quotient is a proper $3$-nilpotent semigroup.  Thus ${\bf S}$ is inherently nondualisable by Theorem \ref{thm:nilvar}.
\end{proof}

\section{Nilpotent groups}
Quackenbush and Szab\'o \cite{quasza} showed that any finite group containing a nonabelian Sylow subgroup is nondualisable, and as observed in \cite{jac03}, their proof in fact shows inherent nondualisability, though it predates the Inherent Nondualisability Lemma~\ref{INDL}.  We now reprove this result by demonstrating an interpretation of ${\bf T}$.  (We mention that Bentz and Mayr \cite{benmay} have recently shown how to obtain this as a special case of a much more general result concerning supernilpotence in congruence modular varieties.)

\begin{thm} \up(Quackenbush and Szab\'o \cite{quasza}.\up)\label{thm:group}
A finite group is inherently nondualisable if it contains a nonabelian Sylow subgroup.
\end{thm}
\begin{proof}
Let ${\bf G}$ be any finite group containing a nonabelian Sylow subgroup.  We first show that we may assume extra conditions on ${\bf G}$ without loss of generality.  Indeed, it will suffice to prove inherent nondualisability for a minimal subgroup of ${\bf G}$ containing a nonabelian Sylow subgroup, in which case, we may assume that ${\bf G}$ is its own Sylow subgroup, of order $p^k$ for some $k>1$ and prime $p$.  Moreover, we can assume that ${\bf G}$ is $2$-generated and nilpotency class $2$; indeed, if the nilpotency class is $k>2$, then there are elements $a,c$ such that $e:=[a,c]\in Z({\bf G})\backslash\{1\}$, and we may consider the subgroup generated by $\{a,c\}$.  

We apply Theorem \ref{thm:T} with $\mathsf{a}=a$, $\mathsf{c}=c$, $\mathsf{e}=[a,c]=e$ and $\mathsf{b}=\mathsf{d}=\mathsf{f}=f$.  Observe that in a nilpotent group of nilpotency class $2$, the term reduct to the commutator operation is a proper $3$-nilpotent semigroup: all commutator products of length $3$ equal $1$, yet there are elements $a,c$ with $[a,c]=e\neq 1$.  So in fact we are in the same situation as in the proof of Theorem \ref{thm:nilvar}, except that as the commutator is not the fundamental operation, we need to revisit the proof that the ghost element $g=1_{\infty}^e$ is not in the subgroup ${\bf A}$ of ${\bf G}^{\mathsf{P}_\infty}$ generated by the line generators $1_{\infty,L}^{a,a}$ and $1_{\infty,L}^{c,c}$ (for lines $L$).

%
%

Consider any $h\in A$ with the property that $h(q)=1$ for all points $q\in \mathsf{P}$.  We show that $h(\infty)=1$ also, showing that $g\notin A$ as required.

First consider $h$ written as a product of line generators.  Each line generator is built over a line from $\mathcal{P}$ (with each line giving rise to two generators) so we may let $L_1,\dots,L_k$ be an enumeration of the lines involved in expressing $h$ as a product of line generators.  Using the law $xy=[x,y]yx$ (which holds for any group) and centrality of commutators (which holds in ${\bf G}$ and in ${\bf G}^{\mathsf{P}_\infty}$ as they are nilpotent of class 2), we may arrange the line generators in this product so that the line generators over $L_1$ appear first, followed by the line generators over $L_2$ and so on, up to the line generators over $L_k$, followed by a product of commutators arising from applications of $xy=[x,y]yx$.  The order of appearance of line generators over any individual line $L_i$ does not change.  For example, the product $(1_{\infty,L_1}^{a,a})(1_{\infty,L_3}^{a,a})(1_{\infty,L_1}^{c,c})(1_{\infty,L_1}^{a,a})(1_{\infty,L_2}^{c,c})$ (with $k=3$) would become 
\begin{multline*}
(1_{\infty,L_1}^{a,a})(1_{\infty,L_1}^{c,c})(1_{\infty,L_1}^{a,a})(1_{\infty,L_2}^{c,c})(1_{\infty,L_3}^{a,a})\\
[(1_{\infty,L_3}^{a,a}),(1_{\infty,L_1}^{c,c})][(1_{\infty,L_3}^{a,a}),(1_{\infty,L_1}^{a,a})][(1_{\infty,L_3}^{a,a}),(1_{\infty,L_2}^{c,c})].
\end{multline*}
(Of course, commutators such as $[(1_{\infty,L_3}^{a,a}),(1_{\infty,L_1}^{a,a})]$ will equal the constant sequence equal to $1$ on all coordinates, however we ignore this and will eventually show that $h$ itself is constantly equal to $1$ also.)
In the case of $h$, this product will be abbreviated as 
\[
w_1w_2\dots w_k v_{1,2}v_{1,3}\dots v_{1,k} v_{2,3}\dots v_{k-1,k},
\]
where $w_i$ is a product of line generators over $L_i$, while $v_{i,j}$ is a (possibly empty) product of commutators between the line generators $L_i$ and $L_j$ (where $i\neq j$: because we do not change the order of appearance of different line generators over the same line $L_i$, all commutators  produced in the rearrangement involve two distinct lines).  Notice that $w_i(q)=1$ unless $q\in L_i$, while $v_{i,j}(q)=1$ unless $q=L_i\wedge L_j$.  Also, for $q\in L_i$ we have $w_i(q)=w_i(\infty)$ and for $q=L_i\wedge L_j$, we have $v_{i,j}(q)=v_{i,j}(\infty)$.  We now show that in fact $w_i(\infty)=v_{i,j}(\infty)=1$ also.

For each $i=1,\dots,k$, let $q_i$ be any point on $L_i$ that is not on the other lines $L_1,\dots,L_{i-1},L_{i+1},\dots,L_k$.  At any point $q\in \mathsf{P}$, the value of $w_i(q)$ is either $1$ (if $q\notin L_i$) or $w_i(q)=w_{i}(q_i)$.  However, because $v_{j,k}(q_i)=1$ for any $j\neq k$ (because $q_i$ is on  the line $L_i$ only) we have $w_{i}(q_i)=h(q_i)=1$.  Thus, as $w_i(\infty)=w_i(q_i)=1$, we have that $w_i$ is constantly equal to $1$.   As $i$ was arbitrary, we have that the product $w_1\dots w_k$ is also constantly equal to~$1$.  Next, for each $i\neq j$, let $q_{i,j}:=L_i\wedge L_j$.  Then $1=h(q_{i,j})=v_{i,j}(q_{i,j})=v_{i,j}(\infty)$.  At all other points $q\neq q_i$ we have $v_{i,j}(q)=1$ as well, thus $v_{i,j}$ and hence $v_{1,2}v_{1,3}\dots v_{1,k} v_{2,3}\dots v_{k-1,k}$ is constantly $1$.  Thus $h$ is the constant sequence equal to $1$, as claimed.
\end{proof}

We now observe that Theorem \ref{thm:group} extends to the semigroup variety setting for trivial reasons.
\begin{thm}\label{thm:groupvar}
Let ${\bf S}$ be a finite semigroup whose variety contains a group with a nonabelian Sylow subgroup.  Then ${\bf S}$ contains a subgroup with a nonabelian Sylow subgroup and hence is inherently nondualisable by Theorem \ref{thm:group}.
\end{thm}
\begin{proof}
It is well known that the variety generated by a finite group ${\bf G}$ contains a group with a nonabelian Sylow subgroup if and only if ${\bf G}$ has a nonabelian subgroup.  This can be proved directly, but can also be seen to be a consequence of results such as Ol$'$\u{s}hanski\u{\i}'s classification of when a group generates a residually large variety (which is if and only if it has a nonabelian Sylow subgroup \cite{ols}).  Next, use another well known fact: that the groups in the semigroup variety generated by a finite semigroup ${\bf S}$ are precisely the groups in the variety generated by the subgroups of ${\bf S}$.  To see why this is true, note that one may find a term $u$ that, under any evaluation of the variables in $u$ inside ${\bf S}$, takes values in the minimum ideal of the subsemigroup ${\bf S}$ generated by the variable interpretation.  (This is roughly the idea used in the proof of Theorem \ref{thm:nilvar}.)  Then group equations may be expressed by replacing variables $x$ by expressions of the form $u^\omega x u^\omega$.  Thus if the variety of ${\bf S}$ contains a group with a nonabelian Sylow subgroup, then ${\bf S}$ itself contains a subgroup with a nonabelian Sylow subgroup.
\end{proof}

\section{Completely simple semigroups}
The general theory of completely simple semigroups easily implies that a completely simple semigroup is isomorphic to the direct product of a group with a rectangular band if and only if the product of any two idempotents is idempotent (equivalently, idempotents form a subsemigroup), which in the periodic case is equivalent to satisfaction of the identity $(xy)^p\approx x^py^p$, where $p$ is the period.  In general, a regular semigroup in which the idempotents form a subsemigroup is known as an \emph{orthodox semigroup}.

We now show that any completely simple semigroup that fails to be orthodox interprets ${\bf T}$ in a way that enables application of Theorem \ref{thm:T}.
\begin{thm}\label{thm:cs}
Let $M[{\bf G},Q]$ be a completely simple semigroup that is not orthodox \up(equivalently, is not isomorphic to a direct product of ${\bf G}$ with a rectangular band\up).  Then $M[{\bf G},Q]$ is inherently nondualisable.
\end{thm}
\begin{proof}
In this case it is routine to show that $M[{\bf G},Q]$ contains a completely simple subsemigroup of the form ${\bf R}:=M[{\bf C},Q]$, where  ${\bf C}$ is a finite cyclic group with generator $\gamma$ and identity $1$ and $Q$ is the matrix 
$
\left(\begin{matrix} 1&1\\
1&\gamma\end{matrix}\right)
$   (see Sapir \cite{sap80} for example, but otherwise, just select any two idempotent elements whose product is not idempotent, and use these to generate a subsemigroup of $M[{\bf G},Q]$).
We show that ${\bf R}$ is inherently nondualisable by interpreting ${\bf T}$.
Let $\mathsf{a}:=(1,1,2)$, $\mathsf{b}=\mathsf{d}=(1,1,1)$, $\mathsf{c}=(2,1,1)$, so that $\mathsf{e}=(1,\gamma,1)$ and $\mathsf{f}=(1,1,1)$.

To apply Theorem \ref{thm:T} we need to show that the ghost element $g=\mathsf{f}_\infty^\mathsf{e}$ cannot be generated by line generators.  Let $h$ be an element of ${\bf R}^{\mathsf{P}_{\infty}}$ that can be obtained as a product of line generators and has $h(q)=\mathsf{f}$ for all points $q\in \mathsf{P}$.  We show that $h(\infty)=\mathsf{f}$ also, showing that $h\neq g$.  Theorem \ref{thm:T} then implies inherent nondualisability.  Let $\underline{\mathsf{f}}$ denote the element of $\{\mathsf{f}\}^{\mathsf{P}_\infty}$, which is the constant sequence equal to $\mathsf{f}$ on all coordinates.  We are going to show that $h=\underline{\mathsf{f}}$.

Let $\ell$ (built over $L$) be the first line generator in some product equalling $h$.  For a point $q\in L$, we have $h(q)=\mathsf{f}$, so that $\ell$ is the generator $\mathsf{d}_{\infty, L}^{\mathsf{b},\mathsf{a}}$ rather than $\mathsf{d}_{\infty,L}^{\mathsf{c},\mathsf{c}}$.  Similarly, if $\ell'$ (over the line $K$) denotes the final line generator involved in a product giving $h$, then $\ell'$ is $\mathsf{d}_{\infty,L}^{\mathsf{c},\mathsf{c}}$.

Next, observe that for any pair of lines $L,K$ we have $\mathsf{b}_{\infty,L}^{\mathsf{a},\mathsf{a}}\mathsf{b}_{\infty,K}^{\mathsf{a},\mathsf{a}}=\mathsf{b}_{\infty,K}^{\mathsf{a},\mathsf{a}}$, while $\mathsf{d}_{\infty,L}^{\mathsf{c},\mathsf{c}}\mathsf{d}_{\infty,K}^{\mathsf{c},\mathsf{c}}=\mathsf{d}_{\infty,L}^{\mathsf{c},\mathsf{c}}$.  Thus any product equalling $h$ can be assumed to alternate between line generators of the form $\mathsf{b}_{\infty,L}^{\mathsf{a},\mathsf{a}}$ and those of the form $\mathsf{d}_{\infty,L}^{\mathsf{c},\mathsf{c}}$ (where the line $L$ varies).  For lines $K,L$ let $w_{K,L}$ denote the product $\mathsf{b}_{\infty,K}^{\mathsf{a},\mathsf{a}}\mathsf{d}_{\infty,L}^{\mathsf{c},\mathsf{c}}=\mathsf{f}_{\infty,K\wedge L}^{\mathsf{e},\mathsf{e}}$.  If $K\neq L$ then $w_{K,L}=\mathsf{f}_{\infty,q}^{\mathsf{e},\mathsf{e}}$, where $q$ denotes $K\wedge L$.  If $K=L$, then $w_{K,L}=\mathsf{f}_{\infty,L}^{\mathsf{e},\mathsf{e}}$.  Let $w_L$ denote $w_{L,L}=\mathsf{f}_{\infty,L}^{\mathsf{e},\mathsf{e}}$ and $w_q=\mathsf{f}_{\infty,q}^{\mathsf{e},\mathsf{e}}$ (which is $w_{K,L}$ for any lines $K,L$ with $K\wedge L=q$).  

As each $w_{K,L}$ lies in the (abelian) group $\{(1,\gamma',1)\mid \gamma'\in {\bf C}\}^{\mathsf{P}_{\infty}}$, elements of the $w_{K,L}$ commute with each other.
Thus the observations so far imply that $h$ can be written as $(w_{q_1})^{n_1}\dots (w_{q_k})^{n_k}(w_{L_1})^{m_1}\dots(w_{L_{k'}})^{m_{k'}}$ where $k,k'$ are nonnegative integers, $n_1,\dots,n_k,m_1,\dots,m_{k'}$ are positive integers, $q_1,\dots,q_k$ are some pairwise distinct points, and $L_1,\dots,L_{k'}$ are some pairwise distinct lines.  For $i\leq k'$, let $p_i$ denote a point on $L_i$ but not equal to $q_j$ for any $j\leq k$, and not on any line $L_j$ for $j\leq k'$ and $i\neq j$.  We show that all $n_i$ and $m_i$ are multiples of $p$, showing that $h$ equals $\underline{\mathsf{f}}\neq g$, as required.  

Let $p$ denote the exponent of ${\bf C}$ (which is the order of $\gamma$, a generator for ${\bf C}$).   We begin by showing that each $m_i$ is a multiple of $p$.

Let $i\leq k'$.  The choice of $p_i$ guarantees that $h(p_i)=(w_{L_i}(p_i))^{m_i}=\mathsf{e}^{m_i}$.  But $h(p_i)=\mathsf{f}$, so that $m_i$ is a multiple of $p$ as claimed.  Thus $w_{L_i}^{m_i}=\mathsf{f}^{\mathsf{P}_{\infty}}$ and so can be ignored in the product representation of $h$.  As $i\leq k'$ was arbitrary, we can now assume without loss of generality that $k'=0$.  That is, $h=(w_{q_1})^{n_1}\dots (w_{q_k})^{n_k}$.

Now let $i\leq k$, and consider the point $q_i$.  The choice of $q_i$ guarantees that $h(q_i)=(w_{q_i}(q_i))^{n_i}=\mathsf{e}^{n_i}$.  But $h(q_i)=\mathsf{f}$ showing that $n_i$ is a multiple of $p$ as claimed.  This completes the proof that $h=\underline{\mathsf{f}}\neq g$.
 Hence, by Theorem \ref{thm:T} we have that ${\bf R}$ is inherently nondualisable.
\end{proof}

\begin{remark}\label{remcs}
Al Dhamri \cite{aldthesis} has shown that a completely simple semigroup is dualisable when it is isomorphic to the direct  product of a dualisable group with rectangular band.  Thus if it is true that every finite group whose Sylow subgroups are abelian is dualisable, then a finite completely simple semigroup is dualisable if and only if it is orthodox and has only abelian Sylow subgroups.
\end{remark}

By Theorem \ref{thm:c0s} and Theorem \ref{thm:nilvar}, a dualisable completely $0$-simple semigroup  $M^0[{\bf G},Q]$ must be such that $Q$ has all entries in $G$.  Then Theorem \ref{thm:cs} and the inherent nondualisability of finite groups with nonabelian Sylow subgroups \cite{quasza} (or see Theorem \ref{thm:group} above) give the following corollary.
\begin{cor}\label{cor:c0s}
If a completely 0-simple semigroup $M^0[{\bf G},Q]$ is dualisable, then ${\bf G}$ has all Sylow subgroups abelian, and $M^0[{\bf G},Q]$ is orthodox.  Equivalently, $M^0[{\bf G},Q]$ is isomorphic to the semigroup obtained by adjoining a zero element to the direct product of ${\bf G}$ \up(with all Sylow subgroups abelian\up) with a rectangular band of the same dimensions as $Q$.
\end{cor}
\begin{remark}
If every group with abelian Sylow subgroups is dualisable, then the result suggested in Remark \ref{remcs} implies that the converse to Corollary \ref{cor:c0s} also holds.  Indeed, if $p$ is the period of the group ${\bf G}$, and $M[{\bf G},Q]$ is orthodox, then $M[{\bf G},Q]\models (xy)^px\approx x$.  Thus the term $t(x,y):=(xy)^px$ is a projection for $M[{\bf G},Q]$, and the result of Davey and Knox \cite{davkno} implies that if $M[{\bf G},Q]$ is dualisable, then so is $M^0[{\bf G},Q]$.
\end{remark}

Let ${\bf L}^1$ denote the 3-element semigroup formed by adjoining an identity element to the two element left zero semigroup.  Let ${\bf R}^1$ denote the corresponding right zero semigroup with adjoined identity.  It is shown in \cite{jac03} that both ${\bf L}^1$ and ${\bf R}^1$ are inherently nondualisable.  In \cite[Proposition 3.5]{pet:13} Petrich shows that a completely regular semigroup  fails to be a normal band of groups if and only if it contains ${\bf L}^1$ or ${\bf R}^1$ as a subsemigroup.  Thus a dualisable completely regular semigroup must be a normal band of groups.  Using Theorem \ref{thm:regular}, we obtain the following corollary to Corollary \ref{cor:c0s}. 
\begin{cor}\label{cor:regular}
A dualisable regular semigroup ${\bf S}$ must be a normal band of groups, with each $\mathcal{J}$-class isomorphic to the direct product of a rectangular band with a group whose Sylow subgroups are abelian.
\end{cor}

Finally we mention that Theorem \ref{thm:cs} extends to variety membership for trivial reasons.
\begin{thm}\label{thm:csvar}
Let ${\bf S}$ be a finite semigroup whose variety contains a nonorthodox completely simple semigroup.  Then the quasivariety of ${\bf S}$ also contains a nonorthodox completely simple semigroup, hence ${\bf S}$ is inherently nondualisable by Theorem~\ref{thm:cs}.
\end{thm}
\begin{proof}
Let ${\bf A}$ be the nonorthodox completely simple semigroup in the variety of ${\bf S}$.  Without loss of generality, we may assume ${\bf A}$ is finite and we may select a finite semigroup ${\bf B}$ in the quasivariety of ${\bf S}$ such that there is a surjective homomorphism from ${\bf B}$ onto ${\bf A}$.  Now let $I$ be the minimal ideal of ${\bf B}$.  We claim that $I$ is a nonorthodox completely simple semigroup.  

To see why $I$ has this structure, select any $s\in A$ that is the image under $\nu$ of some element $s'\in I$.  Now, $s$ divides every element $t\in A$, but in ${\bf B}$, the element $s'$ only divides elements of $I$ (as $I$ is the minimum ideal).  Thus every element of $t\in A$ is the image under $\nu$ of some element of $I$.  Therefore ${\bf A}$ is  quotient of $I$, which is a completely simple semigroup (by the remark immediately following Theorem \ref{thm:Rees}).  As ${\bf A}$ fails the identity $(xy)^p\approx x^py^p$ for $p$ the period of ${\bf S}$, so also does $I$.  Hence $I$ is nonorthodox.
\end{proof}

\section{Other small inherently nondualisable semigroups}
In \cite{jac03} the author showed that the following three $3$-element semigroups are inherently nondualisable.
\[
\parbox{3cm}{\begin{tabular}{c|ccc}
$\cdot$&$1$&$a$&$0$\\
\hline
$1$&$1$&$a$&$0$\\
$a$&$a$&$0$&$0$\\
$0$&$0$&$0$&$0$\\
\end{tabular}\\
\begin{center}${\bf C}_{2,1}^1$\end{center}
}\qquad
\parbox{3cm}{\begin{tabular}{c|ccc}
$\cdot$&$1$&$a$&$b$\\
\hline
$1$&$1$&$a$&$b$\\
$a$&$a$&$a$&$a$\\
$b$&$b$&$b$&$b$\\
\end{tabular}
\begin{center}${\bf L}^1$\rule{0cm}{.57cm}\end{center}
}\qquad
\parbox{3cm}{\begin{tabular}{c|ccc}
$\cdot$&$1$&$a$&$b$\\
\hline
$1$&$1$&$a$&$b$\\
$a$&$a$&$a$&$b$\\
$b$&$b$&$a$&$b$\\
\end{tabular}
\begin{center}${\bf R}^1$\rule{0cm}{.57cm}\end{center}
}
\]
The following theorem parallels Theorems~\ref{thm:nilvar}, \ref{thm:groupvar} and \ref{thm:csvar}.
\begin{thm}\label{thm:small}
\begin{enumerate}
\item If ${\bf S}$ is a finite semigroup whose variety contains ${\bf C}_{2,1}^1$, then the variety of ${\bf S}$ contains a proper $3$-nilpotent semigroup, hence ${\bf S}$  is inherently nondualisable.
\item If ${\bf S}$ is a finite semigroup whose variety contains ${\bf L}^1$ or ${\bf R}^1$, then the quasivariety of ${\bf S}$ contains ${\bf L}^1$ or ${\bf R}^1$ and so is inherently nondualisable.
\end{enumerate}
\end{thm}
\begin{proof}
(1)
Assume that ${\bf C}_{2,1}^1$ is in $\mathsf{HSP}({\bf S})$.  In the subalgebra of ${\bf C}_{2,1}^1\times {\bf C}_{2,1}^1$ on $\{(1,a),(a,1), (a,a),(1,0),(0,1),(0,0)\}$, the set $\{(1,0),(0,1),(0,0)\}$ is an ideal $I$ and the corresponding quotient is a proper $3$-nilpotent semigroup.  So the claim follows from Theorem \ref{thm:nilvar}.  (In only a few more lines this can also be proved using Proposition 17 of \cite{jac03}.)

(2) Next assume that ${\bf L}^1$ is in the variety of ${\bf S}$.  So there is a finite ${\bf M}\in\mathsf{SP}({\bf S})$ and  a surjective homomorphism $\nu:{\bf M}\to {\bf L}^1$.  Let $p$ be the period of ${\bf M}$.  Again, we may let $\mathsf{e}$ be an idempotent element of $\nu^{-1}(1)$ and consider an element $\mathsf{a}$ with $\mathsf{a}=\mathsf{e}\mathsf{a}\mathsf{e}$ from intersection of the minimum ideal of ${\bf M}$ with $\nu^{-1}(a)$.  By replacing $\mathsf{a}$ with $\mathsf{a}^p$ for some power $p$ if necessary, we may assume that $\mathsf{a}$ is idempotent.  Next select any $\mathsf{b}'$ from the intersection of $\nu^{-1}(b)$ with the minimum ideal of ${\bf M}$ and such that $\mathsf{e}\mathsf{b}'\mathsf{e}=\mathsf{b}'$.
Finally, select $\mathsf{b}:=(\mathsf{b}'\mathsf{a})^p$, which is idempotent and has $\mathsf{b}\mathsf{a}=\mathsf{b}$.  As this minimal ideal is a completely simple semigroup (see Section \ref{sec:sgp}: the minimal ideal is isomorphic to a Rees matrix semigroup) it follows that $\mathsf{a}\mathsf{b}=\mathsf{a}$ as well.  Then the subsemigroup on $\{\mathsf{e},\mathsf{a},\mathsf{b}\}$ is isomorphic to ${\bf L}^1$.  The case of ${\bf R}^1$ is by symmetry.
\end{proof}

We mention that ${\bf L}^1$ and ${\bf R}^1$ are orthodox completely regular semigroups.  These provide a further restriction on potential dualisability that can be added to Corollary~\ref{cor:regular}.

\section{Instability of dualisability under direct products}\label{sec:PQ}
We  now give an example of two dualisable semigroups whose direct product is inherently nondualisable.
First consider the semigroup ${\bf M}$ on $\{a,b,e,f,0\}$ whose nonzero products are $ea=a$, $bf=b$, $ee=e$ and $ff=f$.  The subsemigroup on $\{a,e,0\}$ is often denoted by ${\bf P}$, while the subsemigroup on $\{b,f,0\}$ is often denoted by ${\bf Q}$. These semigroups arise mysteriously in a number of algorithmic issues for varieties of semigroups (see Kharlampovich and Sapir \cite{khasap}), but also in other issues that might possibly have a relationship to dualisability, such as in \cite{golsap}.  The semigroup ${\bf M}$ generates the same quasivariety as the direct product ${\bf P}\times {\bf Q}$ as it embeds both ${\bf P}$ and ${\bf Q}$ and is very easily seen to be isomorphic to the subsemigroup of ${\bf P}\times {\bf Q}$ on $(a,0),(e,0),(0,b),(0,f),(0,0)$.   
\begin{thm}\label{thm:PQ}
The semigroup ${\bf P}\times {\bf Q}$ generates a variety containing a proper $3$-nilpotent semigroup.  Hence it is inherently nondualisable.
\end{thm}
\begin{proof}
The second claim will follow from the first and Theorem \ref{thm:nilvar}.  Observe that in the square ${\bf M}^2$, the elements $(a,f), (e,b)$ generate a subsemigroup ${\bf S}$ in which the only elements failing to have a zero coordinate are $(a,f), (e,b), (a,b)$.  Thus the set $I:=S\backslash\{(a,f), (e,b), (a,b)\}$ is an ideal and the Rees quotient ${\bf S}/I$ is very easily verified to be a proper 3-nilpotent semigroup.  Hence ${\bf M}$ is inherently nondualisable by Theorem \ref{thm:nilvar}.
\end{proof}
\begin{theorem}
${\bf P}$ and ${\bf Q}$ are dualisable.
\end{theorem}
\begin{proof}  We apply the IC Duality Theorem \ref{IC} to ${\bf P}$, with the result for ${\bf Q}$ following by symmetry (as ${\bf Q}$ is anti-isomorphic to ${\bf P}$).  We construct an alter ego $\mathbb{P}$ on the set $\{0,a,e\}$.  First observe that ${\bf P}$ is entropic: multiplication itself is a homomorphism from ${\bf P}^2$ to ${\bf P}$, and we include it in the signature of $\mathbb{P}$; we write it as concatenation.  A second homomorphism from ${\bf P}^2$ to ${\bf P}$ we denote by $\wedge$ and is a flat semilattice operation, with $0$ as the bottom element, and $a$, $e$ incomparable elements at height $1$.  We also include the two constants $0$ and $e$ in the signature as well as the idempotent partial operation $\vee$ with domain $\{0,a\}^2\cup\{(e,e)\}$ with $0\vee a=a\vee 0=a$.

We now consider any $\mathbb{X}\leq \mathbb{P}^n$ and any $\alpha:\mathbb{X}\to \mathbb{P}$.  To apply the IC Duality Theorem \ref{IC} we must show that $\alpha$ coincides with a term function of ${\bf P}$.  First note that $\{0,e\}$ is a subuniverse of $\mathbb{P}$ term equivalent to the two element semilattice with constants.  This is well known to form a dualising alter ego for the semilattice $\langle \{0,e\};\cdot\rangle$.  Thus if $\alpha(X)=\{0,e\}$, then we may use semilattice duality to find a term $t$ with $t(x)=\alpha(x)$ for any $x\in X\cap \{0,e\}^n$.  But in this situation, $\alpha(x)=\alpha(x)\alpha(x)=\alpha(x^2)$ and $x^2\in\{0,e\}^n$ for all $x\in X$.  Thus $t(x^2)=\alpha(x)$ for all $x\in X$, as required.  So now we assume that $\alpha(X)=\{0,e,a\}$.

As $\mathbb{X}$ is a finite semilattice with respect to the operation $\wedge$, we may select a smallest element $e^\wedge$ of $X$ that $\alpha$ maps to $e$ and an element $a^\wedge$ that is smallest with respect to mapping to $a$.  Let $I:=\{i\leq n\mid a^\wedge (i)=a\}$, which is non-empty, as elements in $X\cap \{0,e\}^n$ are mapped by $\alpha$ into $\{0,e\}$ because they are beneath the constant $\underline{e}$ with respect to the order induced by $\wedge$.  Let $J=\{j\leq n\mid a^\wedge(j)=e\}$, which may possibly be empty.  Let $j_1,\dots,j_k$ be an enumeration of the elements in $J$.   For each $i\in I$ let $t_i(x)$ be the term $x_ix_{j_1}^2\dots x_{j_k}^2$ (which is simply $x_i$ if $J$ is empty).  We claim that there is $i\in I$ such that $\alpha(x)=t_i(x)$.  First observe that $\alpha(a^\wedge e^\wedge)=ae=a$, so that $e^\wedge(j)=e$ for $j\in I\cup J$.  Thus for any $x\in \alpha^{-1}(\{a,e\})$ we have $\alpha(x)=t_i(x)$, for any $i\in I$.

Now let us assume for contradiction that for every $i\in I$ there is a $z^{(i)}$ such that $\alpha(z^{(i)})=0$ but $t_i(z^{(i)})\neq 0$.  Note that the only nonzero products in ${\bf P}$ are of the form $ae\dots e$ or $e\dots e$.  Thus if $t_i(z^{(i)})\neq 0$ it follows that $a^\wedge(j)=e$ implies $z^{(i)}(j)=e$.  If $t_i(z^{(i)})=e$ we may replace $z^{(i)}$ by $a^\wedge z^{(i)}$ (we will keep the notation $z^{(i)}$ for this possibly new choice).  Thus we may assume that $t_i(z^{(i)})=a$.  Then $t_i(a^\wedge\wedge z^{(i)})=a\wedge a=a$, so that we may replace $z^{(i)}$ by $a^\wedge\wedge z^{(i)}$ (again, keeping the notation $z^{(i)}$ for this possibly new choice).  Summarising the properties of the elements $z^{(i)}$ (for $i\in I$) we have 
\[
z^{(i)}(j)\in\begin{cases}
\{a\}&\text{ if }i=j\\
\{0,a\}&\text{ if }j\in I\\
\{e\}&\text{ if }j\in J\\
\{0\}&\text{ otherwise}
\end{cases}
\]
Such elements are in the domain of $\vee$ on $\mathbb{X}$, giving $\bigvee_{i\in I}z^{(i)}=a^\wedge$ yielding the contradiction $0=0\vee 0\vee \dots \vee 0=\alpha(\bigvee_{i\in I}z^{(i)})= \alpha(a^\wedge)=a$.  Hence we conclude that there is $i\in I$ such that $t_i(x)=\alpha(x)$ for all $x\in X$, showing that $\mathbb{P}$ dualises $\mathbf{P}$ by the IC Duality Theorem \ref{IC}.
\end{proof}

We conclude with some open problems, which should guide future directions in the goal of a classification of dualisable semigroups.  
\begin{problem}\label{problem1}
When is a finite Clifford semigroup dualisable?
\end{problem}
Clearly, it is necessary that all subgroups have only abelian Sylow subgroups.  But is this sufficient?  A number of families of Clifford semigroups over abelian groups were shown to be dualisable by Nadia Al Dhamri in her thesis \cite{aldthesis}.  Problem \ref{problem1} is also of interest in the monoid signature (and in the inverse semigroup signature, but the inverse is term definable from multiplication in a finite Clifford semigroup, so the inverse semigroup theoretic version coincides with the straight semigroup version).  Even the restriction of Problem \ref{problem1} to commutative semigroups is of interest.  A broader target would be to characterise dualisability for regular semigroups.  Must such a semigroup be a normal band of dualisable groups?

In \cite{DPW}, Davey, Pitkethly and Willard showed that a finite algebra generating a residually large congruence meet semidistributive variety is inherently nondualisable.  Semigroups do not typically generate congruence meet semidistributive varieties, however, all  nondualisable semigroups discovered to date generate residually large varieties and all semigroups generating residually large varieties for which the dualisability question has been resolved, are inherently nondualisable (see Golubov and Sapir \cite{golsap}, Kublanovsky \cite{kub} or McKenzie \cite{mck} for a classification of residually large semigroup varieties, while results in the present article imply all known nondualisability results for finite semigroups).
\begin{problem}\label{problem2}
Is it true that a finite semigroup generating a residually large variety is inherently nondualisable?
\end{problem}
The present article already covers a number of the cases required to complete a solution to Problem \ref{problem2}.  We note also, that all of the currently known nondualisable finite semigroups are now known to be \emph{inherently} nondualisable: the one example left unresolved in \cite{jac03} (see second example on page 488) contains both ${\bf P}$ and ${\bf Q}$ as subsemigroups, and hence is inherently nondualisable by Theorem \ref{thm:PQ}.
\begin{problem}\label{problem3}
Is there a nondualisable but not inherently nondualisable finite semigroup?
\end{problem}


\begin{thebibliography}{99}
\bibitem{AlD}  Al Dhamri, N.: Dualities for quasivarieties of bands. Semigroup Forum {\bf 88}, 417--432 (2014)
\bibitem{aldthesis}  Al Dhamri N.: Natural Dualities for Quasi-varieties of Semigroups.  PhD Thesis, La Trobe University, Melbourne (2013)
\bibitem{BDPW} Bentz, W.,  Davey, B.A., Pitkethly, J.G., Willard, R.: Dualizability of automatic algebras. J. Pure Appl. Algebra {\bf 218}, 1324--1345 (2014)
\bibitem{benmay} Bentz, W., Mayr, P.: Supernilpotence prevents dualizability.  J. Austral. Math. Soc. {\bf 96}, 1--24 (2014)
\bibitem{cladav98}
 Clark, D.M., Davey, B.A.: Natural Dualities for the Working Algebraist. Cambridge University Press, Cambridge (1998)
\bibitem{DILM} Davey, B.A., Idziak, P.M., Lampe W.A., McNulty, G.F.: Dualizability and graph algebras. Discrete Math. {\bf 214}, 145--172 (2000)
 \bibitem{davkno} Davey, B.A.,   Knox, B.J.: Regularising natural dualities. Acta Math. Univ. Comenianae. {\bf 68}, 295--318 (1999)
 \bibitem{DPW} Davey, B.A.,  Pitkethly, J.G.,  Willard, R.: Dualisability versus residual character: a theorem and a counterexample. J. Pure Appl. Algebra {\bf 210}, 423--435 (2007)
 \bibitem{davwer} Davey, B.A.,  Werner, H.: Dualities and equivalences for varieties of algebras, In: Contributions to Lattice Theory (Szeged, 1980) (A.P. Huhn and E.T. Schmidt, eds), pp. 101--275.  Colloq. Math. Soc. J\'anos Bolyai 33, North-Holland (1983)
 \bibitem{golsap} Golubov, E.A.,   Sapir, M.V.: Varieties of finitely approximable semigroups. Soviet
Mathematics {\bf 20}, 828--832 (AMS translation, 1979)
 \bibitem{HKMST} Hall, T.E., Kublanovsky, S.I., Margolis, S., Sapir, M.V., Trotter, P.G.: Decidable and undecidable problems related to finite 0-simple semigroups. J. Pure Appl. Algebra {\bf 119} 75--96 (1997) 
 \bibitem{HMS} Hofmann, K.H., Mislove, M., Stralka, A.: The Pontryagin Duality of Compact 0-Dimensional Semilattices and its Applications. Springer (1974).
\bibitem{how}  Howie, J.M.: Fundamentals of Semigroup Theory, 2nd edition. Oxford University Press, New York (1995)
\bibitem{kub} Kublanovski, S.I.: Finite approximability of prevarieties of semigroups with respect to
predicates.  In: Modern Algebra (Gos. Ped. Inst., Leningrad, 1980), pp. 58--88 [Russian]
\bibitem{jac03} Jackson, M., Dualisability of finite semigroups. Internat. J. Algebra Comput. {\bf 13},  481--497 (2003)
\bibitem{jactro} Jackson M.,  Trotta, B.: The division relation: congruence conditions and axiomatisability. Commun. Algebra {\bf 38}, 534--566 (2010)
\bibitem{jacvol09} Jackson, M.,  Volkov, M.: Relatively inherently nonfinitely q-based semigroups. Trans. Amer. Math. Soc. {\bf 361}, 2181--2206 (2009)
\bibitem{jacvol09b} Jackson, M., Volkov,  M.: Undecidable problems for completely 0-simple semigroups.  J. Pure Appl. Algebra {\bf 213}, 1961--1978 (2009)
\bibitem{khasap}  Kharlampovich, O.G., Sapir,  M.V.: Algorithmic problems in varieties.  Internat. J. Algebra Comput. {\bf 5}, 379--602 (1995)
\bibitem{klerotspe}  Kleitman, D.J., Rothschild,  B.R.,  Spencer, J.H.: The number of semigroups of order $n$. Proc. Amer. Math. Soc. {\bf 55}, 227--232 (1976)
\bibitem{mck} McKenzie, R.: Residually small varieties of semigroups.  Algebra Universalis {\bf 13},
171--201 (1981)
\bibitem{nic}  Nickodemus, M.H.: Natural Dualities for Finite Groups with Abelian Sylow Subgroups. PhD thesis, University of Colorado (2007)
\bibitem{ols} Ol$'$\u{s}anski\u{\i}, A.Ju.: Varieties of finitely approximable groups.  Izv. Akad. Nauk. SSSR
Ser. Mat. {\bf 33}, 915--927 (1969) [Russian; English translation in Math. USSR Izv. {\bf 3}, 867--877 (1969)]
\bibitem{pet:13} Petrich, M.: Characterizing some completely regular semigroups by their subsemigroups. J. Austral. Math. Soc. {\bf 94}, 397--416 (2013)
\bibitem{quasza} Quackenbush, R.W.,  Szab\'o, Cs.: Nilpotent groups are not dualizable.  J. Austral. Math. Soc. {\bf 72}, 173--180 (2002)
\bibitem{quasza02b} Quackenbush, R.W.,  Szab\'o, Cs.: Strong duality for metacyclic groups. J. Austral. Math. Soc. {\bf 73}, 377--392 (2002)
\bibitem{sap80} Sapir, M.V.: On the quasivarieties generated by finite semigroups. Semigroup Forum {\bf 20},
73--88 (1980)
\end{thebibliography}
\end{document}